\documentclass[11pt]{article}
\usepackage{graphicx,verbatim}
\usepackage{lineno}
\usepackage{lipsum,pdflscape}
\usepackage{amsmath,amssymb}
\usepackage{xcolor,colortbl}
\usepackage{xspace}
\usepackage{color}
\usepackage{epsfig}
\usepackage[algosection,lined,ruled,linesnumbered]{algorithm2e}
\usepackage{subfigure}

\graphicspath{{./Fig/}}
\newtheorem{theorem}{Theorem}[section]

\newtheorem{corollary}[theorem]{Corollary}
\newtheorem{lemma}[theorem]{Lemma}
\newtheorem{remark}[theorem]{Remark}
\newtheorem{definition}[theorem]{Definition}
\newtheorem{example}[theorem]{Example}

\newcommand{\qed}{\hfill $\square$\medskip}

\begin{document}

\title{More on the $2$-restricted optimal pebbling number}

\author{
Saeid Alikhani$^{}$\footnote{Corresponding author}
\and
Fatemeh Aghaei
}

\date{\today}

\maketitle

\begin{center}
Department of Mathematical Sciences, Yazd University, 89195-741, Yazd, Iran\\
{\tt alikhani@yazd.ac.ir, aghaeefatemeh29@gmail.com}
\end{center}


\begin{abstract}
Let $G=(V,E)$ be a simple graph. 
 A function $f:V\rightarrow \mathbb{N}\cup \{0\}$ is called a configuration of pebbles on the vertices of $G$  and the the weight of $f$ is  $w(f)=\sum_{u\in V}f(u)$  which is  just the total number of pebbles assigned to vertices. A pebbling step from a vertex $u$ to one of its
 neighbors $v$  reduces $f(u)$ by two and increases $f(v)$  by one. A pebbling configuration $f$ is said to be solvable if for every vertex $ v $, there exists a sequence (possibly empty) of pebbling moves that results in a pebble on $v$. 
A pebbling configuration $f$ is a $t$-restricted pebbling configuration (abbreviated $t$RPC) if $f(v)\leq t$ for all $v\in V$. The  $t$-restricted optimal pebbling number $\pi_t^*(G)$  is  the minimum weight of a solvable $t$RPC on $G$. 
 Chellali et.al. [Discrete Appl. Math. 221 (2017) 46-53] characterized connected
 graphs $G$ having small $2$-restricted optimal pebbling numbers and characterization of graphs $G$ with $\pi_2^*(G)=5$ stated as an open problem. In this paper, we solve this problem. We improve the upper bound of the $2$-restricted optimal pebbling number of trees of order $n$.
Also we study $2$-restricted optimal pebbling number of some grid graphs,  corona and neighborhood corona of two specific graphs.

\end{abstract}

\noindent{\bf Keywords:} Roman domination, optimal pebbling number, pebbling number, corona.

\medskip
\noindent{\bf AMS Subj.\ Class.}:  05C99. 

\section{Introduction and definitions}
 Let $G=(V,E)$ be a simple graph of order $n$. 
 A function $f:V\rightarrow \mathbb{N}\cup \{0\}$ is called a configuration of pebbles on the vertices of $G$  and the quantity $w(f)=\sum_{u\in V}f(u)$
 is called the weight of $f$ which is  just the total number of pebbles assigned to vertices. A pebbling step from a vertex $u$ to one of its neighbors $v$  reduces $f(u)$ by two and increases $f(v)$  by one. A pebbling configuration $f$ is said to be solvable, if for every vertex $ v $, there exists a sequence (possibly empty) of pebbling moves that results in a pebble on $v$. The pebbling number $ \pi(G) $ equals the minimum number $k$ such that every pebbling configuration $ f $ with $ w(f) = k $ is solvable.
 
The configuration with a single pebble on every vertex
except the target shows that $\pi(G)\geq n$, while the configuration with
$2^{ecc(r)}-1$  pebbles on the farthest vertex from $r$, and no
pebbles elsewhere, shows that $\pi(G)\geq 2 ^{{\rm diam}(G)}$ when $r$ is
chosen to have $ecc(r)={\rm diam}(G)$.
As usual let $Q^d$ be the $d$-dimensional hypercube  as the graph on all binary $d$-tuples,
pairs of which that differ in exactly one coordinate are
joined by an edge. Chung \cite{Chung} proved that $\pi(Q^d)=2^d$. Graphs $G$ like
$Q^d$  which have $\pi(G)=|V(G)|$  are  known as Class $0$.
The terminology comes from a lovely theorem of Pachter,
Snevily, and Voxman \cite{Pachter}, which states that if ${\rm diam}(G)=2$, 
then $\pi(G)\leqslant n+1$. 
Therefore, there are two classes of diameter two graphs, Class $0$ and
 Class $1$. The
Class $0$ graphs are $2$-connected.

Pachter et al. \cite{Pachter}, defined the optimal pebbling number $ \pi^{*}(G) $ to be the minimum weight of a solvable pebbling configuration of $G$. A solvable pebbling configuration of $G$ with weight $ \pi^{*}(G) $ is called a $ \pi^{*}$-configuration. The decision problem associated with computing the optimal pebbling number was shown to be NP-Complete in \cite{Milans}.

 Chellali et al. in \cite{Chellali}, introduced a generalization of the optimal pebbling number. A pebbling configuration $  f $ is a $ t $-restricted pebbling configuration (abbreviated $ t$RPC), if $ f(v) \leqslant t $ for all $ v \in V $. They defined the $ t $-restricted optimal pebbling number $ \pi^{*}_{t}(G) $ as the minimum weight of a solvable $t$RPC  on $G$. If $f$ is a solvable $ t$RPC  on $G$ with $ w(f)=\pi^{*}_{t}(G) $, then $f$ is called a $\pi^{*}_{t}$-configuration of $G$. We note that the limit of $t$ pebbles per vertex applies only to the initial configuration. That is, a pebbling move may place more than $t$ pebbles on a vertex.

We use the following terminology.  The open neighborhood of a vertex $ v \in V $ is the set $ N(v)=\lbrace u \mid uv \in E\rbrace $ of vertices adjacent to $ v $, and its closed neighborhood is $N[v] = N(v) \cup \lbrace v\rbrace$. The open neighborhood of a set $ S \subset V $ of vertices is the set $ N(S)=\bigcup_{v\in S} N(v) $, while the closed neighborhood of a set $S$ is the set $ N[S]=\bigcup_{v\in S} N[v] $. The degree of a vertex $v$ is $ \deg(v) = |N(v)|$. The subgraph of $G$ induced by a set of vertices $ S $ is denoted by $ G[S]$.

A set $ S \subset V $ is a dominating set of a  $ G $, if every vertex in $ V\setminus S $ is adjacent to at least one vertex in $ S $, and $ S $ is a total dominating set, if every vertex in $ V $ is adjacent to at least one vertex in $ S $. The domination number $ \gamma(G) $ (respectively, total domination number $ \gamma_{t}(G) $) of a graph $ G $ is the minimum cardinality of a dominating set (respectively, total dominating set) in $G$.

Characterization of connected
 graphs $G$ having $\pi_2^*(G)=k$ for small $k$ is an interesting problem and  has solved for $k=2,3,4$ as follows.
 
 \begin{theorem} {\rm \cite{Chellali}}\label{1}
 Suppose that $G$ is a nontrivial connected graph. Then
 \begin{enumerate}
 	\item[(i)] $ \pi^{*}_{2}(G)=2 $ if and only if $ \gamma(G)=1 $ and $ \gamma_{t}(G)=2 $. 
 	
 	\item[(ii)]
   $ \pi^{*}_{2}(G)=3 $ if and only if $ \gamma(G)=\gamma_{t}(G)=2 $.
 	
 	\item[(iii)] 
 	$\pi_2^*(G)=4$ if and only if $\gamma_t(G)\geq 3$ and there exist two vertices $u$ and $v$ such that $\{u,v\}\cup (N(u)\cap N(v))$ dominates $G$. 
 	
 \end{enumerate}
 \end{theorem}

\medskip
 Characterization of graphs $G$ with $\pi_2^*(G)=5$ stated as an open problem in \cite{Chellali}. In the next section, we solve this problem. In Section 3,   we obtain an upper bound for the $2$-restricted optimal pebbling number
of trees. In Section 4, we study 2-restricted optimal pebbling number of some grid graphs and finally in Section 5, we obtain the $2$-restricted optimal pebbling number for corona and neighborhood corona of two specific graphs. 

\section{Characterization of graphs $G$ with $\pi^{*}_{2}(G) = 5$ } 

As stated in \cite{Chellali}, for every nontrivial graph $G$ of order $n$, $ 2 \leqslant \pi^{*}_{2}(G) \leqslant n $. Clearly, the upper bound is attained if and only if every component of $ G $ is $ K_{1} $ or $ K_{2} $. On the other hand, the lower bound is attained for any graph of order $ n $ having a vertex of degree $ n-1 $, such as complete graphs $ K_{n} $ ($ n\geqslant 2 $) and wheels $ W_{1,n} $ ($ n\geqslant 3 $).

 \begin{figure}[ht]
\centering
\includegraphics[scale=.5]{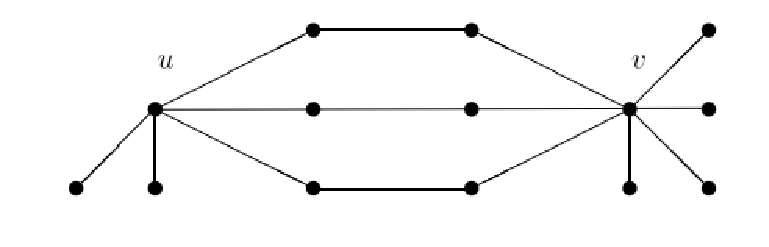}
\caption{The graph $ G_{3,2,4} $ }\label{4}
\end{figure}
 A consequence  of Theorem \ref{1}    involves graphs with no triangles and no induced $ C_{4} $. Let $\mathcal{F} $ be the family of graphs $ G_{i,j,k} $ formed from two vertices $ u $ and $ v $ by adding $ i $ disjoint paths of length $ 3 $ between vertices $ u $ and $ v$, ~$j $ vertices of degree $1$ adjacent to $ u $, and $ k $ vertices of degree $ 1 $ adjacent to $ v $, such that $ i \geqslant 1, i + j \geqslant 2 $, and $ i + k\geqslant 2 $. As an example, the graph $ G_{3,2,4} $ in $\mathcal{F} $ is given in Figure \ref{4}. In fact, two vertices $ u $ and $ v $ dominate graph $ G\in \mathcal{F} $.
 
 We need the following results: 
 \begin{corollary}{\rm\cite{Chellali}}\label{2}
 Let $G$ be a nontrivial connected graph of girth at least five. Then $ \pi^{*}_{2}(G)=4 $ if and only if $ \gamma_{t}(G)=3 $ or $G \in \mathcal{F}$.
 \end{corollary}
 \begin{theorem}{\rm\cite{Chellali}}\label{3}
 If $T$ is a nontrivial tree of order $ n $, then $ \pi_{2}^{*}(T)\leqslant \lceil \frac{5n}{7} \rceil $.
 \end{theorem}
 
 In what follows, we characterize connected graphs $ G $ having $ \pi^{*}_{2}(G)=5 $.
 \begin{theorem}
 Let $G$ be a nontrivial connected graph. Then
 $ \pi^{*}_{2}(G)=5 $ if and only if $ \gamma_{t}(G)\geqslant 4 $ and while for any $ u, v\in V(G) $, $  \lbrace u, v \rbrace\cup(N(u)\cap N(v))\ $ cannot dominate $ G $, one of the following conditions happens:
 \begin{itemize}
 \item[i)]
 There are three vertices $ u, v $ and $ w\in N(u)\cap N(v) $ such that $ \lbrace u, v, w\rbrace\cup(N(u)\cap N(v))\cup(N(u)\cap N(w))\cup(N(v)\cap N(w))$ dominates $ G $.
 \item[ii)]
 There are  $ u , v $ and $ w\in N(v) $ such that $ \lbrace u, v, w\rbrace\cup(N(u)\cap N(v))\cup(N(u)\cap N(w))$ dominates $ G $.
 \item[iii)]
 There are  $ u , v $ and $ w\in N(N(u)\cap N(v)) $ such that $ \lbrace u, v, w\rbrace\cup(N(u)\cap N(v))$ dominates $ G $.
 \end{itemize}  
 \end{theorem}
 \begin{proof}
 Assume that  $ \gamma_{t}(G)\geqslant 4 $ and for any $ u, v\in V(G) $, $  \lbrace u, v \rbrace\cup(N(u)\cap N(v))\ $ cannot dominate $ G $, so by Theorem \ref{1}, $ \pi^{*}_{2}(G)\geqslant 5 $. To show that $ \pi^{*}_{2}(G)=5 $, we give a solvable $ 2RPC $ with weight $5$. Clearly, according to the conditions, it suffices to   consider $ f(u)=2 $, $ f(v)=2 $ and $ f(w)=1$.

Conversely, let $ \pi^{*}_{2}(G)=5 $. By Corollary \ref{2}, we have  $ \gamma_{t}(G)\geqslant 4 $. We shall assert that one of the three conditions is satisfied. Let $ f $ be a $ \pi_{2}^{*} $-configuration on $ G $. If five vertices are assigned one pebble, then clearly $ G $ has order $5$. Since by Theorem \ref{3}, $ \pi_{2}^{*}(G)\leqslant \pi_{2}^{*}(T) $ for spanning tree $ T $, so for every graph $ G $ of order $5$, $ \pi_{2}^{*}(G)\leqslant 4 $  and this is a contradiction.

Hence, either three vertices are assigned one pebble and one vertex is assigned two pebbles, or two vertices are assigned two pebbles each and one vertex are assigned one pebble. We  consider these two possibilities for $ f $.

 \begin{figure}[ht]
\centering
\includegraphics[scale=0.8]{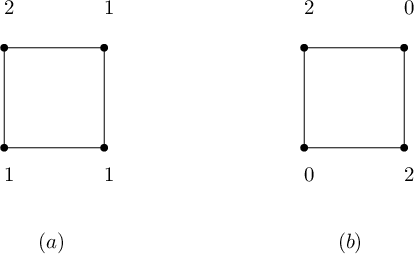}
\caption{there are $ u, v $ that $  \lbrace u, v \rbrace\cup(N(u)\cap N(v))\ $ dominates $ G $.}\label{13}
\end{figure}

Case 1) Let $ S=\lbrace u, v, w, h\rbrace  $ such that $ f(u)=2, f(v)=f(w)=f(h)=1 $, and $ f(x)=0 $ for all $ x\in V\setminus S$. Since $ f $ is a $ \pi_{2}^{*} $-configuration of $ G $, it follows that $ S $ dominates $ G $. Therefore any vertex $ V\setminus S $ is adjacent to $ S $. Since $ G $ is connected, so either vertices $ v, w, $ and $h $ are adjacent to vertex $ u $ or there is a path from $ u $ to them such that contain at most one internal vertex without pebble. By new distribution, $ G[S] $ is connected. Note that $ G[S] $ does not have $4$-cycle, otherwise there are $ u, v $ that $  \lbrace u, v \rbrace\cup(N(u)\cap N(v))\ $ dominates $ G $ (see Figure \ref{13}). Therefore three cases (a), (b) or (c) in Figure \ref{12} will be obtained. Therefore we can put the pebbles such that  Case 2 occurs (see Figure \ref{12}).

 \begin{figure}[h!]
\centering
\includegraphics[scale=0.7]{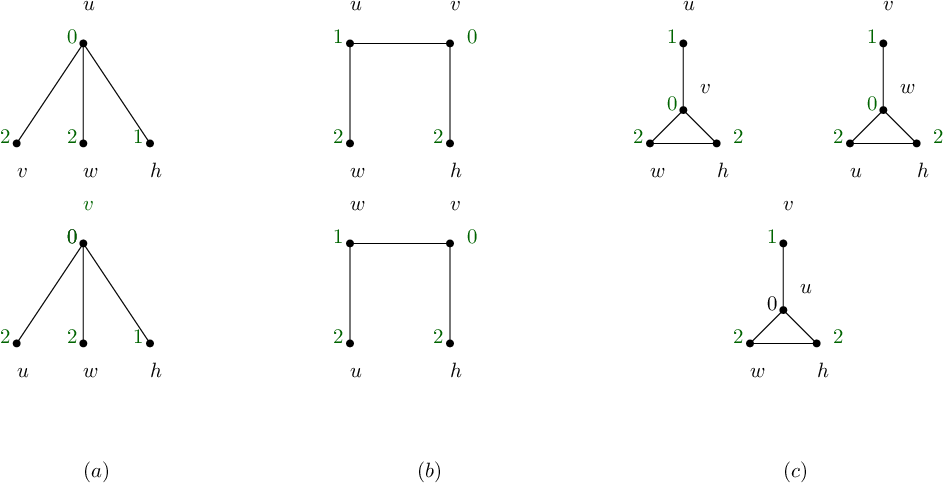}
\caption{New distribution of $ G[S] $. }\label{12}
\end{figure}

Case 2)  Let $ S=\lbrace u, v, w\rbrace  $ such that $ f(u)=f(v)=2, f(w)=1 $, and $ f(x)=0 $ for all $ x\in V\setminus S$. Since $ f $ is a $ \pi_{2}^{*} $-configuration of connected graph $ G $, so $ w\in N(u)\cup N(v)\cup N(N(u)\cap N(v))$, otherwise $ w $ is adjacent to a vertex say $ x $ which  is dominated by $  \lbrace u, v \rbrace\cup(N(u)\cap N(v))$, so one can change to $ f(w)=0, f(x)=1$. Therefore, the graph $ G[S] $ becomes one of the graphs  in Figure \ref{11}. Note that in  Figure \ref{11},  we see  the equivalences of $(a)$ and $(i)$, $(b)$ and $(ii)$, and finally  $(c)$  and $(iii)$. 
Therefore, the result follows. \qed
 \end{proof}

 \begin{figure}[h!]
\centering
\includegraphics[scale=0.8]{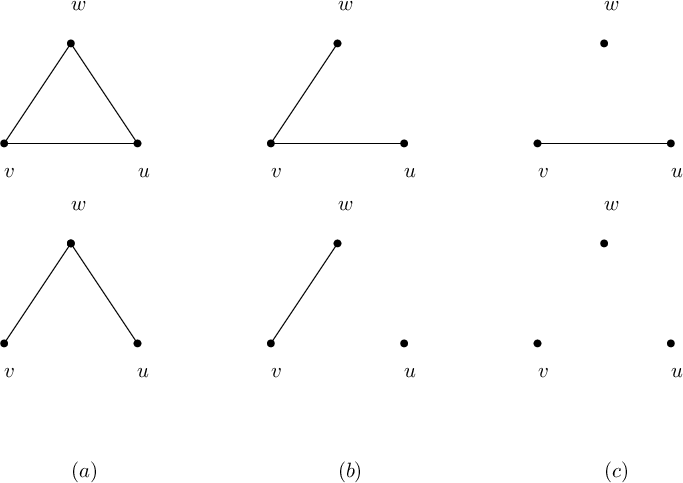}
\caption{$ f(u)=f(v)=2, f(w)=1 $, and $ f(x)=0 $ for all $ x\in V\setminus S$.}\label{11}
\end{figure}

We close  this section with an example.

\begin{example}  The Petersen graph $P$ in Figure \ref{5}, has $ \pi^{*}_{2}(P)=5 $.  Note that this assignment is indeed a $ \pi^{*}_{2} $-configuration of $P$.
	\end{example} 
 \begin{figure}[h!]
\centering
\includegraphics[scale=0.8]{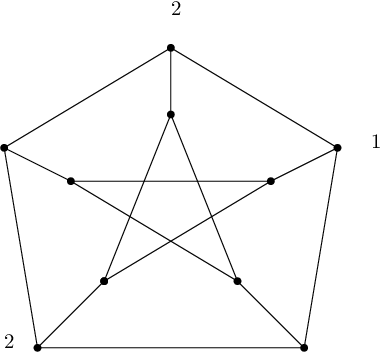}
\caption{The Petersen graph $P$ has $ \pi^{*}_{2}(P)=5 $.}\label{5}
\end{figure}
  \medskip

   \section{ $2$-restricted optimal pebbling number of trees }
   In this section, we obtain an upper bound for the $2$-restricted optimal pebbling number of trees.  As we stated before, Chellali, Haynes, Hedetniemi and Lewis proved that the $2$-restricted optimal pebbling number of trees of order $n$ is at most $\lceil\frac{5n}{7}\rceil$ (Theorem \ref{3}). To improve this upper bound, we need to  consider Roman domination number of a graph. 
   
    For a graph $ G = (V, E) $, let $ f:V\rightarrow \lbrace 0,1,2\rbrace $, and let $ (V_{0},V_{1},V_{2}) $ be the ordered partition of $ V $ induced by $ f $, where $ V_{i}=\lbrace v\in V \vert f(v)=i \rbrace $ and $ \vert V_{i}\vert=n_{i} $, for $ i = 0, 1, 2 $. Note that there exists a one to one correspondence between the functions 
    $f:V\rightarrow \lbrace 0,1,2\rbrace $ and the ordered partitions $ (V_{0},V_{1},V_{2}) $ of $ V $. Thus, we will write $ f=(V_{0},V_{1},V_{2}) $. 
        A function $ f=(V_{0},V_{1},V_{2}) $ is a Roman dominating function (RDF) if $ V_{2}\succ V_{0} $, where $ \succ $ means that the set $ V_{2} $ dominates the set $ V_{0} $, i.e. $ V_{0}\subseteq N[V_{2}] $. The weight of $ f $ is $ f(V)=\sum_{v\in V}f(v)=2n_{2}+n_{1} $.
        The Roman domination number, denoted $ \gamma_{R}(G) $, equals the minimum weight of an RDF of $ G $, and a function $ f=(V_{0},V_{1},V_{2}) $ is a $ \gamma_{R} $-function if it is an RDF and $ f(V)=\gamma_{R}(G) $.
                We assert $ \gamma_{R}(T_{n})\leqslant \gamma_{R}(P_{n}) $ for every tree $T_n$ of order $ n $.  
   We need the following results:

   \begin{lemma}{\rm\cite{Cockayne}}\label{3.1}
   	For any graph $ G $, $\gamma(G) \leqslant\gamma_{R}(G)\leqslant 2\gamma(G)$.
   \end{lemma}
   
   \begin{lemma}{\rm\cite{Cockayne}}\label{3.2}
   	Let $ f=(V_{0},V_{1},V_{2}) $ be a $ \gamma_{R} $-function of an isolate-free graph $G$, such that $|V_1|=n_{1} $ is a minimum. Then
   	\begin{itemize}
   		\item[a)]
   		No edge of $ G $ joins $ V_{1} $ and $ V_{2} $. 
   		\item[b)]
   		$ V_{1}$ is independent, and $ V_{0}\cup V_{2} $ is a vertex cover.
   		\item[c)]
   		$ V_{0}\succ V_{1} $
   		\item[d)]
   		Each vertex of $ V_{0} $ is adjacent to at most one vertex of $ V_{1} $.
   		\item[e)]
   		$ V_{2} $ is a $ \gamma $-set of $ G[V_{0}\cup V_{2}] $.
   	\end{itemize}
   \end{lemma}
   \begin{theorem}\label{3.3}
   	\begin{enumerate}
   	\item[(i)] 
   	{\rm\cite{Cockayne}}
   	For paths $ P_{n} $ of  order $ n\geqslant 3 $, $ \gamma_{R}(P_{n})=\lceil\dfrac{2n}{3}\rceil$.
   	
   	\item[(ii)]{\rm\cite{Chellali}}\label{3.4}
   	For paths $ P_{n} $ of order $ n\geqslant 3 $, $ \pi^{*}_{2}(P_{n})=\lceil\dfrac{2n}{3}\rceil$.
   	
   	\item[(iii)]
   	$ \pi^{*}_{2}(P_{n})=\gamma_{R}(P_{n})=\lceil\dfrac{2n}{3}\rceil$.
   	\end{enumerate} 
   \end{theorem}

    We need the following lemma: 
   \begin{lemma}{\rm\cite{Chellali}}\label{3.5}
   	For any graph $ G $, $ \pi^{*}_{2}(G)\leqslant \gamma_{R}(G) $.
   \end{lemma}

   \begin{theorem}
    For any tree $ T $ of order $ n\geqslant 3 $, $ \gamma_{R}(T)\leqslant \gamma_{R}(P_{n}).$
   \end{theorem}
   \begin{proof}
   	Let $ T_{n} $ be a tree of order $ n $ with $\gamma_{R} $-function $f=(V_{0},V_{1},V_{2}) $ such that $ n_{1}=|V_1| $ is minimum. We consider two following cases:
   	\begin{itemize}
   		\item
   		If $ n\equiv 0 $ or $ 2 $ $(mod\, 3)$, then $ 2\gamma(P_{n})=\gamma_{R}(P_{n}) $. Since  $ \gamma(T_{n})\leqslant\gamma(P_{n}) $  by considering Lemma \ref{3.1}, $ \gamma_{R}(T_{n})\leqslant 2\gamma(T_{n})\leqslant 2\gamma(P_{n})=\gamma_{R}(P_{n}) $. 
   		\item
   		If $ n\equiv 1 $ $(mod\, 3)$, then by Lemma \ref{3.2}, each leaf $ v\in V(T_{n}) $ belongs to $ V_{0}\cup V_{1} $ and so $ f(v)\leqslant 1 $. By removing $ v $, we have $ \gamma_{R}(T_{n})\leqslant\gamma_{R}(T_{n-1})+1 $. Therefore
   		   			$$
   			 \gamma_{R}(T_{n})\leqslant\lceil\frac{2(n-1)}{3}\rceil+1=\lceil\frac{2n}{3}\rceil=\gamma_{R}(P_{n}).
   			$$ \qed  
   		   	\end{itemize}
   \end{proof}
   \begin{corollary}\label{T_n}
   	If $T_{n}$ is a tree of order $ n\geqslant 3 $, then $ \pi_{2}^{*}(T_{n})\leqslant \lceil \frac{2n}{3} \rceil $.
   \end{corollary}
   \begin{proof}
   	By Lemma \ref{3.5}, $ \pi^{*}_{2}(T_{n})\leqslant \gamma_{R}(T_{n}) $, so we have $ \pi_{2}^{*}(T_{n})\leqslant \lceil \frac{2n}{3} \rceil $. \qed
   \end{proof}
   
   \begin{remark}
   The bound in Corollary \ref{T_n} is sharp for paths.
   \end{remark}

   \begin{corollary}\label{last}
   If $ G $ is a connected graph of order $ n\geqslant 3 $, then $ \pi_{2}^{*}(G)\leqslant \lceil \frac{2n}{3} \rceil $.
   \end{corollary}

 Corollary \ref{last} asserts that an upper bound for the $ 2 $-restricted optimal pebbling number of any connected graph $ G $ is the $ 2 $-restricted optimal pebbling number of any spanning tree $ T $ of $ G $.

   \section{$2$-restricted optimal pebbling number of grid graphs}
   In this section, we study $2$-restricted optimal pebbling number of some grid graphs. 
    If a vertex $v$ is reachable under some configuration $f$, we say that $ f $ covers $ v $ or $ v $ is covered by $ f $. The set of all covered vertices by $ f $ is called the coverage of $ f $. The size of a configuration $ f $, expressed as $ \vert f\vert $, is the total number of pebbles used in $ f $, while the size of the coverage by a configuration $ f $, expressed as $ Cov(f) $, is the total number of vertices covered by $ f $.

   A unit is the coverage obtained when $f$ begins as pebbles on a single vertex, called the source. A block is a combination of several units such that the subgraph induced by all reachable edges is connected. We need the following theorem: 
 
 \begin{theorem} {\rm\cite{Xue}} \label{Xue}
 	An optimal configuration $ f $ with $ n $ pebbles on grid graphs $ P_{n}\square P_{2} $ ($ n\neq 2, 5 $) consists of the following two types of fundamental blocks, 2-2 blocks and 2-1 blocks (see Figure \ref{block}).
  
   \begin{figure}[h!]
\centering
\includegraphics[scale=0.8]{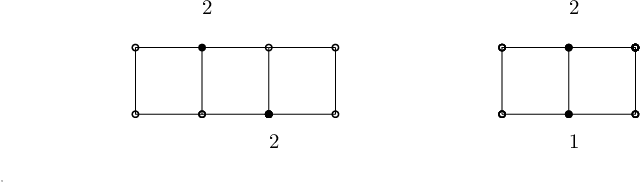}
\caption{A 2-2 block (left) and a 2-1 block (right).}\label{block}
\end{figure}
\end{theorem}
 
 \begin{corollary}
For $ n\neq 2, 5 $,  $ \pi^{*}_{2}(P_{n}\square P_{2})=n $.
 \end{corollary}
\begin{proof}
Since for any graph $ G $, $ \pi^{*}(G)\leqslant\pi^{*}_{2}(G) $, so $ n\leqslant\pi^{*}_{2}(P_{n}\square P_{2}) $ for $ n\neq 2, 5 $. By Theorem \ref{Xue}, $ \pi^{*}_{2}(P_{n}\square P_{2})\leqslant n $ for $ n\neq 2, 5 $. Therefore the result follows. \qed
\end{proof}

It is easy to conclude that $ \pi^{*}_{2}(P_{2}\square P_{2})=3 $ and $ \pi^{*}_{2}(P_{5}\square P_{2})=6 $
(see Figure \ref{P*2}).
   \begin{figure}[h!]
\centering
\includegraphics[scale=0.8]{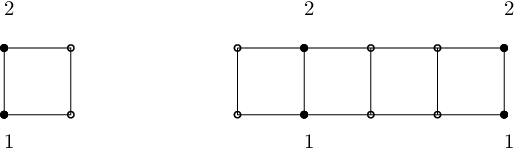}
\caption{2-restricted optimal pebbling configuration of $ P_{2}\square P_{2} $ (left) and $ P_{5}\square P_{2} $ (right).}\label{P*2}
\end{figure}

\begin{theorem} {\rm \cite{Xue}}\label{P3}
For $ n\geqslant 2 $, $ \pi^{*}(P_{n}\square P_{3})=n+1 $.
\end{theorem}
  \begin{figure}[h!]
\centering
\includegraphics[scale=0.8]{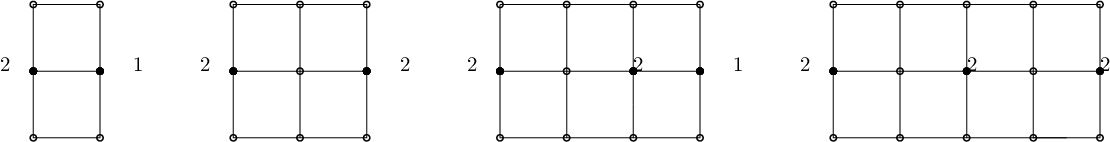}
\caption{2-restricted optimal pebbling configuration of $ P_{n}\square P_{3} $, for $ n=2, 3, 4, 5 $.}\label{8}
\end{figure}

\begin{corollary}
For $ n\geqslant 2 $, $ \pi^{*}_{2}(P_{n}\square P_{3})=n+1 $.
\end{corollary}
\begin{proof}
According to the pebbling configuration pattern given in Figure \ref{8}, we have $ \pi^{*}_{2}(P_{n}\square P_{3})\leqslant n+1 $. Therefore we have the result by Theorem \ref{P3}.
\qed 
\end{proof}

In the rest of this section, we study the covering ratio of a block. First we recall the definition. 
\begin{definition}{\rm \cite{Xue}}
The covering ratio of a configuration $ f $ is given by $ \dfrac{Cov(f)}{\vert f\vert} $, which is the ratio between the size of the coverage and the size of the configuration.
\end{definition}
\begin{remark}\label{5.2}
If $ f $ is a solvable $ 2 $RPC on $G$, then there are two types unit with covering ratio $ 1 $ or $ 2.5 $.
\end{remark}
\begin{figure}[ht]
\centering
\includegraphics[scale=0.8]{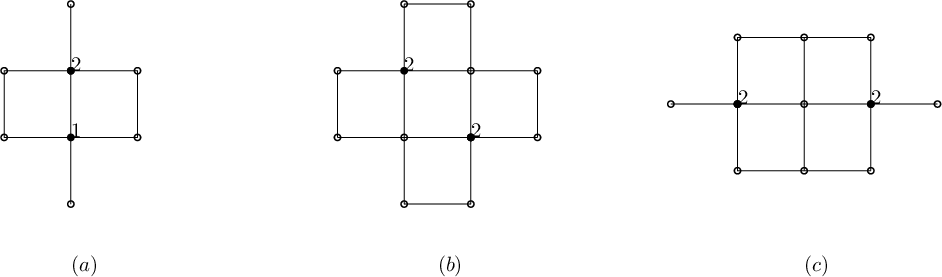}
\caption{The block with two units.}\label{9}
\end{figure}

\begin{lemma}\label{5.3}
	In a $ 2 $RPC, the covering ratio of a block which consists of two units is at most $ 3 $.
\end{lemma}
\begin{proof}
Since the block with two units is one of the cases  in Figure \ref{9}, the covering ratio (which is for (b)) equals 3. \qed
\end{proof}
\begin{lemma}\label{5.4}
On $ P_{n}\square P_{4} $ for $ n\geqslant 4 $, the covering ratio of a block in a $ 2 $RPC is at most $ 3 $.
\end{lemma}
\begin{proof}
We use induction. For the base case, the block that consists of only one source has a covering ratio $ 2.5 $ or $ 1 $ by Remark \ref{5.2}. For the block with two units, the result follows by Lemma \ref{5.3}.
For $ k\geq 1 $, suppose that the covering ratio of a block that consists of $ k $ units is no more than $ 3 $. Note that a block $ B $ with  $ k+1 $ sources can be decomposed into a block $ B' $ with $ k $ sources and a unit $ U $. By the inductive hypothesis,  the covering ratio of $ B' $ and $ U $ is at most $ 3 $. The block  $ B $ can have a covering ratio greater than three,  only if $ B' $ and $ U $ interact (If two blocks have a vertex in common, then we say they interact. The interaction vertices are the vertices covered by both blocks. A block can be thought of as several interacting units). Since the unit $ U $ with one pebble can cover at most three  excess vetrices while $ U $ with two pebbles can cover at most six excess vetrices, so $ \dfrac{Cov(B')+3}{\vert B'\vert +1}\leqslant 3$ and $ \dfrac{Cov(B')+6}{\vert B'\vert +2}\leqslant 3$. Therefore the result follows.\qed
\end{proof}
 
\begin{theorem}
For $ n\geqslant 4 $, $ \pi^{*}_{2}(P_{n}\square P_{4})= $
\newcommand{\threepartdef}{6}
{
$\left\{\begin{array}{lll}
\frac{4n+6}{3};		&	\mbox{if~~}		 n\equiv 0 (mod~3) \\
\frac{4n+5}{3};		&	\mbox{if~~}		 n\equiv 1 (mod~3) \\
\frac{4n+4}{3};		&	\mbox{if~~}		 n\equiv 2 (mod~3) 
\end{array}
\right.$
}.
 \end{theorem}

 \begin{proof}
 According to Lemma \ref{5.3}, the maximum covering ratio occurs, when   all vertices  of  $ P_{n}\square P_{4} $  covered by  block (b) that has at most ratio $ 3 $ and it also covers the bandwidth of the graph. Therefore the weight of solvable $ 2 $RPC is $  \pi^{*}_{2}(P_{n}\square P_{4}) $. 
 Since  every block used in the optimal configuration on $ P_{n}\square P_{4} $ must be sharp-ended, so one of the cases in Figure \ref{10} occurs and therefore we have the result by  patterns. \qed

\begin{figure}[ht]
\centering
\includegraphics[scale=1.1]{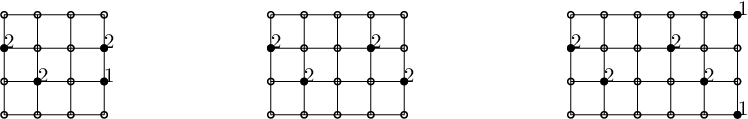}
\caption{2-restricted optimal pebbling configuration of $ P_{n}\square P_{4} $, for $ n=4, 5, 6 $.}\label{10}
\end{figure}
 \end{proof}

We close this section with the following problem:

\noindent {\bf Problem 1}: Determine $\pi^{*}_{2}$ for grid graphs $P_{n}\square P_{m}$, for $n, m\geqslant 5$ ?

\section{2-restricted optimal pebbling in some products of graphs}
In this section, we obtain some results for the  $2$-restricted optimal pebbling of  corona and neighborhood corona of some specific graphs. Let recall the definitions of these two products.  
 The corona product $G\circ H$ of two graphs $G$ and $H$ is defined as the graph obtained by taking one copy of $G$ and $\vert V(G)\vert $ copies of $H$ and joining the $i$-th vertex of $G$ to every vertex in the $i$-th copy of $H$.
 The neighborhood corona product $ G\star H $ of two graphs $ G $ and $ H $ is defined as the graph obtained by taking one copy of $ G $ and $\vert V(G)\vert $ copies of $H$ and joining the neighbours of the $i$-th vertex of $G$
to every vertex in the $i$-th copy of $ H $.

We need the following result: 
\begin{corollary}{\rm \cite{Alikhani}}
For any graph $ H $ and any complete graph $ K_{n} $,
\begin{itemize}
\item[i)]
$ \pi^{*}_{2}(K_{n}\circ H)=4 $ for $ n\geqslant 3 $,
\item[ii)]
$ \pi^{*}_{2}(K_{n}\star H)=3 $ for $ n\geqslant 2 $.
\end{itemize}
\end{corollary}


\begin{theorem}\label{6.2}
For any graph $ H $,
\begin{itemize}
\item[i)]
$ \pi^{*}_{2}(C_{n}\circ H)=n $ for $ n\geqslant 4 $,
\item[ii)]
$ \pi^{*}_{2}(C_{n}\star H)= $ 
\newcommand{\threepartdef}{6}
{
$\left\{\begin{array}{llll}
\frac{3n}{4};		&	\mbox{if~~}		 n\equiv 0 (mod~4) \\
\frac{3n+1}{4};		&	\mbox{if~~}		 n\equiv 1 (mod~4) \\
\frac{3n+2}{4};		&	\mbox{if~~}		 n\equiv 2 (mod~4) \\
\frac{3n+3}{4};		&	\mbox{if~~}		 n\equiv 3 (mod~4) 
\end{array}
\right.$
} ~~for $ n\geqslant 4 $.
\end{itemize}
\end{theorem}
\begin{proof}
Note that in the $2$-restricted optimal pebbling configuration of graph $ G\circ H $ or $ G\star H $, no pebbles are placed on the vertices of the copies of $ H $, because by picking up the pebble(s) of $i$-th copy of $ H $ and putting them on $i$-th vertex one can be reduced or fixed the total size of the configuration.
\begin{itemize}
\item[i)]
We need $2$ pebbles on the $i$-th vertex of $ C_{n} $, then one can reach a pebble on any vertex of $i$-th copy of $ H $. First we present a solvable $2$RPC $ f $ with size $ \vert f\vert=n $. Since there are two non-adjacent  vertices $ u , v\in V(C_{n}) $, so we put two pebbles on each of them. Consider two vertices $ w, z $ from two distinct $ uv $-paths of graph $ C_{n} $ and put a pebble on remain vertices except $ w, z $. Therefore $ \pi^{*}_{2}(C_{n}\circ H)\leqslant n $.

Now, we claim $ \pi^{*}_{2}(C_{n}\circ H)\geqslant n $. Let $ f $ be $2$RPC with size $ n-1 $ on graph $ C_{n} $ arbitrarily, we prove that any vertex of $ V(C_{n}) $ is not $2$-reachable. There is at least one vertex, say  $ w $ without pebbles in $2$RPC. Since $ \deg (v)=2 $ for $ v\in V(C_{n}) $, so every vertex can receive at most two  excess pebbles (at most a pebble from each neighbour) by sequence of pebbling steps (moves). Suppose that $ u, v $ have $ 2 $ pebbles, so there are at least three vertices without pebbles. Since only one vertex without pebbles can be placed in each $ uv $-path and this is not possible, in fact, the number of zeros is always more than the number of vertices with 2 pebbles, which is a contradiction.

\item[ii)]
In this case, we decompose a graph $ C_{n}\star H $ into the distinct block(s) $ A $ and block $ B $ (see Figure \ref{15}). Since $ \pi_{2}^{*}(P_{n})=\lceil\dfrac{2n}{3}\rceil $ and block $ A $ contains the path graph $ P_{4} $, so $ \pi_{2}^{*}(A\star H)=\pi_{2}^{*}(P_{4}\star H)=3 $. If $ V(C_{n})=\lbrace v_{0},v_{1},...,v_{n-1}\rbrace $, then graph $ C_{n}\star H $  consists of the block(s) $ A $ and a type of block $ B $ whose for $ i=1,...,n-2 $, the vertex $ v_{i}\in V(C_{n}) $ is connected to copy $ H_{i-1} $, $ H_{i+1} $ and vertices $ v_{i-1} $, $ v_{i+1} $. Finally $ v_{0} $ is  connected to $ v_{n-1} $,  $ H_{n-1} $ and likewise $ v_{n-1} $ is connected  to $ H_{0} $. According to $2$-restricted optimal pebbling configuration of graph $  \pi_{2}^{*}(B\star H) $ in Figure \ref{15}, the result follows.
\end{itemize}
\end{proof}

\begin{figure}[ht]
\centering
\includegraphics[scale=0.6]{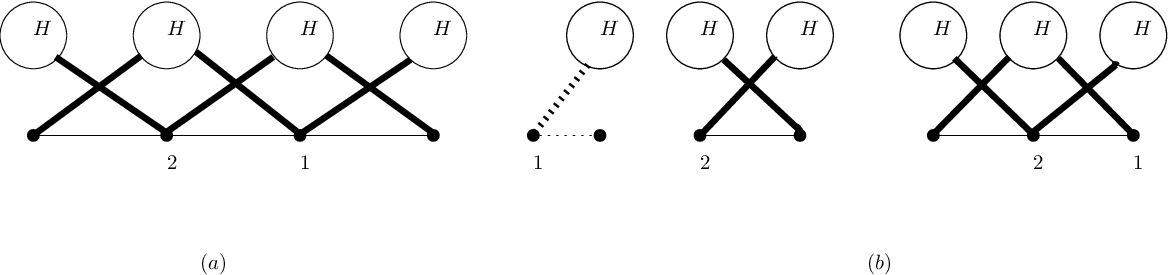}
\caption{(a) block $ A $, (b) Three types of block $ B $}\label{15}
\end{figure} 

\begin{theorem}
For any graph $ H $ and any path graph $ P_{n} $,
\begin{itemize}
\item[i)]
$ \pi^{*}_{2}(P_{n}\circ H)=n+1 $ for $ n\geqslant 4 $,
\item[ii)]
$ \pi^{*}_{2}(P_{n}\star H)=\pi^{*}_{2}(C_{n}\star H) $ for $ n\geqslant 4 $ .
\end{itemize}
\end{theorem}
\begin{proof}
\begin{itemize}
\item[i)]
It is clear that the graph $ P_{n}\circ H $ is obtained by removing an edge from the subgraph $ C_{n} $ of $ C_{n}\circ H $. Let a vertex $ w $ be connected to the vertex $ u $ with $ 2 $ pebbles in a $2$-restricted optimal pebbling configuration of graph $ C_{n}\circ H $. By removing edge $ uw $ and putting a pebble on the vertex $ w $, we have $ \pi^{*}_{2}(P_{n}\circ H)=n+1 $.
\item[ii)]
The proof is similar to  part (ii) of Theorem \ref{6.2} with the difference that the $ v_{0} $ does not connect to $ v_{n-1} $ and  $ H_{n-1} $ and likewise $ v_{n-1} $ does not connect to $ H_{0} $.
\end{itemize}
\end{proof}

\end{document}